\documentclass[12pt]{amsart}
\usepackage{amsmath,amssymb,amsbsy,amsfonts,amsthm,latexsym,
                        amsopn,amstext,amsxtra,euscript,amscd,ifthen}

\usepackage[english]{babel}
\begin{document}
\newtheorem{theorem}{Theorem}[section] 
\newtheorem{lemma}[theorem]{Lemma}     
\newtheorem{corollary}[theorem]{Corollary}
\newtheorem{proposition}[theorem]{Proposition}


\newcommand{\p}{p} 
\newcommand{\s}{s} 
\newcommand{\q}{q} 
\newcommand{\F}[1][\q]{\mathbb{F}_{#1}} 
\newcommand{\K}{\mathbb{K}} 
\newcommand{\Q}{\mathbb{Q}}
\newcommand{\Z}{\mathbb{Z}}
\newcommand{\cS}{\mathcal{S}}
\newcommand{\FX}[1][X]{\mathbb{F}_\q[#1]} 
\newcommand{\KX}[1][X]{\mathbb{K}[#1]} 
\newcommand{\modulo}{\bmod} 
\newcommand{\ff}[1][M]{
  \ifthenelse{\equal{#1}{M}}{f}{} 
  \ifthenelse{\equal{#1}{X}}{f(X)}{} 
  \ifthenelse{\equal{#1}{\root_{i}}}{f(\root_{i})}{} 
  \ifthenelse{\equal{#1}{\beta_{i}}}{f(\beta_{i})}{} 
  \ifthenelse{\equal{#1}{\hh}}{f \circ \hh}{} 
} 
\renewcommand{\gg}[1][M]{
  \ifthenelse{\equal{#1}{M}}{g}{} 
  \ifthenelse{\equal{#1}{X}}{g(X)}{} 
  \ifthenelse{\equal{#1}{\hh}}{g \circ \hh}{} 
} 
\newcommand{\hh}[1][M]{
  \ifthenelse{\equal{#1}{M}}{h}{} 
  \ifthenelse{\equal{#1}{X}}{h(X)}{} 
} 
\newcommand{\degree}[1]{
  \ifthenelse{\equal{#1}{\ff}}{d}{} 
  \ifthenelse{\equal{#1}{\gg}}{e}{} 
  \ifthenelse{\equal{#1}{\hh}}{d_{\hh}}{} 
  \ifthenelse{\equal{#1}{\f{\n}{X}}}{d_{\ff}^{\n}}{} 
  \ifthenelse{\equal{#1}{\f{\n-1}{X}}}{d_{\ff}^{\n-1}}{} 
} 
\newcommand{\n}{n} 
\newcommand{\f}[2]{f^{(#1)}(#2)} 
\newcommand{\g}[2]{g^{(#1)}(#2)} 
\newcommand{\coefficient}[2]{
  \ifthenelse{\equal{#2}{\ff}}{a_{#1}}{}
  \ifthenelse{\equal{#2}{\gg}}{b_{#1}}{}
  \ifthenelse{\equal{#2}{\hh}}{a_{#1}}{}
  \ifthenelse{\equal{#2}{\f{\n}{X}}}{\(a_{\degree{\ff}}\)^{\frac{\degree{\ff}^n-1}{\degree{\ff}-1}}}{}
  \ifthenelse{\equal{#2}{\f{\n-1}{X}}}{\(a_{\degree{\ff}}\)^{\frac{\degree{\ff}^{n-1}-1}{\degree{\ff}-1}}}{}
  \ifthenelse{\equal{#2}{\f{\n2}{X}}}{\(a_{\degree{\ff}}\)^{\frac{\degree{\ff}^{n-1}-1}{\degree{\ff}-1}}}{}
} 
\newcommand{\Cf}[1]{
  \ifthenelse{\equal{#1}{1}}{C_f}{C_{f^{#1}}}
  }
\newcommand{\principal}[1]{\coefficient{\degree{#1}}{#1}} 
\newcommand{\element}[1][\alpha]{#1} 
\newcommand{\Res}[2]{\mathrm{Res}\left(#1,#2\right)} 
\newcommand{\Disc}[1]{\mathrm{Disc}(#1)} 
\newcommand{\Tr}{\mathrm{Tr}}
\renewcommand{\root}{\gamma} 
\renewcommand{\(}{\left (} 
\renewcommand{\)}{\right )}
\newcommand{\fa}[2]{F_{#2}(a_0,\ldots,a_{\degree{#1}})}
\newcommand{\fy}[2]{F_{#2}(Y,a_1,\ldots,a_{\degree{#1}x})}
\def\Orb{\mathrm{Orb}(f)}
\def\Nm{\mathrm{Nm}}
\newcommand{\comm}[1]{\marginpar{%
\vskip-\baselineskip 
\raggedright\footnotesize
\itshape\hrule\smallskip#1\par\smallskip\hrule}}


\title[Stable Polynomials over Finite Fields]
 {Stable Polynomials over Finite Fields} 


\author[D. G\'omez-P\'erez]{Domingo G\'omez-P\'erez} 
\address{Department of Mathematics, University of Cantabria, Santander 39005, Spain}
\email{domingo.gomez@unican.es}
\author[A. P. Nicol\'as]{Alejandro P. Nicol\'as} 
\address{Departamento de Matem‡tica Aplicada, Universidad de Valladolid, Spain}
\email{anicolas@maf.uva.es}
\author[A. Ostafe]{Alina Ostafe} 
\address{Department of Computing, Macquarie University, Sydney NSW 2109, Australia}
\email{alina.ostafe@mq.edu.au}
\author[D. Sadornil]{Daniel Sadornil} 
\address{Department of Mathematics, University of Cantabria, Santander 39005, Spain}
\email{sadornild@unican.es}



\thanks{A. N. was supported by MTM2010-18370-C04-01, A.~O. was
supported by SNSF   Grant 133399 and D. S. was supported by
MTM2010-21580-C02-02 and MTM2010-16051.}

\maketitle

\begin{abstract}
We use the theory of resultants to study the stability of an arbitrary polynomial $f$
over a finite field $\F$, that is, the property of having all its iterates irreducible.
This result partially generalises the quadratic polynomial case described
by R. Jones and N. Boston. Moreover, for $p=3$, we show 
that certain polynomials of degree three are not stable. 
We also use the Weil bound for multiplicative character
sums to estimate the number of stable arbitrary polynomials
over finite fields of odd characteristic.
\end{abstract}

\section{Introduction}
For a polynomial $\ff$ of degree at least 2 and coefficients in a
field $\K$, we define the following sequence:
\begin{equation*}
  \f{0}{X}=X,\quad \f{\n}{X}=\f{\n-1}{\ff[X]},\ \n\ge 1.
\end{equation*}
A polynomial $\ff$ is \textit{stable} if $f^{(n)}$ is irreducible
over $\K$ for all $\n\ge 1$. 
In this article, $\K=\F$ is a finite field 
with $\q$ elements, where $\q=\p^{\s}$ and
$\p$ an odd prime.

Studying the stability of a polynomial is an exciting problem which has
attracted a lot of attention. However, only few results are known and 
the problem is far away from being well understood.

The simplest case, when the polynomial is quadratic, has been studied
in several works. For example, some results concerning the stability
over $\F$ and $\Q$ can be found in \cite{Ali,Ayad,Danielson,Jon,Jones09}. 
In particular, by~\cite[Proposition~2.3]{Jones09},
a quadratic monic polynomial $f \in \K[X]$ over a field $\K$ 
of odd characteristic and 
with the unique critical point $\gamma$, is stable
if the set
$$\{-f(\gamma)\} \cup \{f^{(n)}(\gamma)\mid n\ge 2\}$$ contains no squares.
In the case when
$\K=\mathbb{F}_q$ is a finite field of odd characteristic,
this property is also necessary.

In~\cite{GomezNicolas10} an
estimate of the number of stable quadratic polynomials over the
finite field $\F$ of odd characteristic is given, while in~\cite{ALOS} 
it is proved that almost
all monic quadratic polynomials $f\in\Z[X]$ are stable over $\Q$.
Furthermore, in~\cite{ALOS} it is shown that there are no stable quadratic
polynomials over finite fields of characteristic two. One might
expect that this is the case over any field of characteristic two,
which is not true as is also shown in~\cite{ALOS} where an
example of a stable quadratic polynomial over a function field of
characteristic two is given.

The goal of this paper is to characterize the set of stable
polynomials of arbitrary degree and to devise a test for checking
the stability of polynomials.

Our techniques come from  theory of resultants and they use the relation between 
irreducibility of polynomials and the properties of the discriminant of polynomials. 
Using these techniques, we partially generalize previous results 
known  for quadratic polynomials.

A test for stability of quadratic polynomials was given
in~\cite{Ostafe09}, where it was shown that checking the
stability of such polynomials can be done in time $q^{{3/4}+o(1)}$. 
As in ~\cite{Jones09}, for an arbitrary polynomial $f$ over $\mathbb{F}_q$, 
the set defined by $$\{\f{\n}{\root_1}\ldots\f{\n}{\root_{k}}\  \mid \ n\geq 1\ \},$$
where $\root_i$, $i=1,\ldots,k$, are the roots of the 
derivative of the polynomial $\ff$,
plays also an important role in checking the stability of $f$. 
In particular, we use techniques based on resultants of polynomials
together with the Stickelberger's theorem to prove our
results. 
We introduce analogues of the
orbit sets defined in~\cite{Jones09} for arbitrary degree $d\ge 2$ polynomials.
As in~\cite{Ostafe09}, 
we obtain a nontrivial estimate for the cardinality of these sets for polynomials with irreducible derivative. 
We also give an estimate for the
number of stable arbitrary polynomials which generalises the result
obtained in~\cite{GomezNicolas10} for quadratic stable polynomials.

The outline of the paper is the following: in Section~\ref{sec:preliminaries} 
we introduce the preliminaries necessary to understand the paper.
These include basic results about resultants 
and discriminants of polynomials. This section ends with the Stickelberger's result. 
Next, Section~\ref{sec:stabilityPolynomials}
is devoted to proving a necessary condition for the stability of a  polynomial. 
We define a set, which generalizes the orbit set for
a quadratic polynomial, and then we give an upper bound
on the number of elements of this set. 
Section~\ref{sec:nonExistence} gives  a new proof of the 
result that appeared in~\cite{ALOS} for cubic polynomials when the characteristic is 
equal to 3.
Finally, in Section~\ref{sec:numberStables} we give an estimate of the number of stable 
polynomials for any degree. 
For that, we relate the number of stable polynomials with estimates of 
certain multiplicative character sums. 

\section{Preliminaries}
\label{sec:preliminaries}
Before proceeding with the main results, it is necessary to introduce some 
concepts related
to commutative algebra. Let $\K$ be any field and let $\ff\in\KX$ be
a polynomial of degree $d$
with leading coefficient $\principal{\ff}$. The \textit{discriminant} of $\ff$,
denoted by $\Disc{\ff}$, is defined by
\begin{equation*}
  \Disc{\ff}=\principal{\ff}^{2d-2}\prod_{i<j}(\alpha_i-\alpha_j)^2,
\end{equation*}
where $\alpha_1,\ldots, \alpha_{d}$ are the roots of $\ff$ in
some extension of $\K$.
\medskip
\newline
It is widely known that  for any polynomial $\ff\in\KX$, its discriminant
is an element of the field $\K$.
Alternatively, it is possible to compute $\Disc{f}$ using resultants.
We can define the resultant of two polynomials $\ff$ and $\gg$
over $\K$ of degrees $d$ and $e$, respectively,
with leading coefficients $a_d$ and $b_e$, as
\begin{equation*}
  \Res{\ff}{\gg}=
  \principal{\ff}^{e}\principal{\gg}^{d}\prod (\alpha_i-\beta_j),
\end{equation*}
where $\alpha_i, \beta_j$ are the roots of $\ff$ and $\gg$, respectively.

Like the discriminant, the resultant belongs to $\K$.
In the following lemmas we summarize several known results about resultants without proofs. The
interested reader can find them in~\cite{Cox07,LN97}.
\begin{lemma}
  \label{lem:resultant_eval_root}
  Let $\ff,\gg\in\KX$ be two polynomials of degrees $d\ge 1$ and $e\ge 1$
  with leading coefficients $a_d$ and $b_e$, respectively. Let 
  $\beta_1,\ldots,\beta_{\degree{\gg}}$ be the roots of $\gg$ in an extension field of $\K$.
  Then,
  \begin{equation*}
  \Res{\ff}{\gg}=
   (-1)^{d\degree{\gg}}
  \principal{\gg}^{d}\prod_{i=1}^{\degree{\gg}} \ff[\beta_{i}].
  \end{equation*}
\end{lemma}
The behaviour of the resultant with respect to the multiplication is given
by the next result.
\begin{lemma}
\label{lem:resultant_multiplication}
  Let $\K$ be any field. Let $\ff,\gg,\hh\in\KX$ be polynomials of degree greater than 1 and
  $a\in\K$.
  The following  hold:
  \begin{align*}
  &  \Res{\ff\gg}{\hh}=\Res{\ff}{\hh}\Res{\gg}{\hh},&\Res{ a\ff}{\gg}=a^{\degree{\gg}}\Res{\ff}{\gg},
  \end{align*}
 where $\deg g=e$.
\end{lemma}
The relation between $\Disc{\ff}$ and $\Res{\ff}{\ff'}$ is given by the next statement.
\begin{lemma}
  \label{lem:resultant_discriminant}
  Let $\K$ be any field and $\ff\in\KX$ be a polynomial of degree $d\geq
  2$ with leading coefficient $a_d$, non constant derivative $\ff'$ and $\deg \ff'=k\le d-1$. Then, we have the relation
  \[\Disc{\ff}= \Cf{1}\Res{\ff}{\ff '},\]
  \noindent where $\Cf{1} =(-1)^{\frac{d(d-1)}{2}}\principal{\ff}^{d-k-2}$.
\end{lemma}

One of the main tools used to prove our main result regarding the stability of arbitrary polynomials
is the Stickelberger's result~\cite{Stickelberger97},
which gives the parity of the number of distinct irreducible factors of a polynomial over a finite field of odd characteristic.

\begin{lemma}
\label{thm:Stickelberger}
  Suppose $\ff\in\FX$, $q$ odd, is a polynomial of degree $d\ge 2$ and is
  the product of $r$ pairwise
  distinct irreducible polynomials over $\F$.
 Then $r\equiv d\modulo 2$ if and only
  if $\Disc{\ff}$ is a square in $\F$.
\end{lemma}

\section{Stability of arbitrary polynomials}
\label{sec:stabilityPolynomials}
In this section we give a necessary condition for the  stability of arbitrary
polynomials. For this purpose,  we use the
following general result known as Capelli's Lemma, see~\cite{FS}.
\begin{lemma}
\label{lem:Capelli} Let $\K$ be a field, $f,g\in\KX$, and let
$\beta\in\overline{\K}$ be any root of $g$. Then $g(f)$ is
irreducible over $\K$ if and only if both $g$ is irreducible over
$\K$ and $f-\beta$ is irreducible over $\K(\beta)$.
\end{lemma}

We prove now one of the main results about the stability of an arbitrary polynomial.
We note that our result partially generalises the quadratic polynomial
case presented in~\cite{Jones09} which is known to be necessary
and sufficient over finite fields.

\begin{theorem}
  \label{thm:criterium}
  Let $q=p^s$, $p$ be an odd prime, and $\ff\in\FX$ a stable polynomial of degree $d\geq
  2$ with leading coefficient $a_d$, non constant derivative $\ff'$ and $\deg \ff'=k\le d-1$. Then the following hold:
  \begin{enumerate}

   \item if $d$ is even, then $\Disc{f}$
     and $\principal{\ff}^{k}\Res{f^{(n)}}{\ff '}$, $n\ge 2$, are nonsquares in $\F$;
   \item if $d$ is odd, then $\Disc{f}$
     and $(-1)^{\frac{d-1}{2}}\principal{\ff}^{(n-1)k+1}\Res{f^{(n)}}{\ff '}$, $n\ge 2$, are squares in $\F$.
   \end{enumerate}
\end{theorem}
\begin{proof}
Let $\ff\in\FX$ be a stable polynomial. We assume first that
$d$ is even. We have that $f^{(n)}$ is irreducible for
any $n$, and thus, by  Capelli's Lemma~\ref{lem:Capelli}, we know
that  $f-\alpha$ is irreducible over
$\mathbb{F}_{q^{d^{n-1}}}$, where $\alpha$ is a root of $f^{(n-1)}$. By
Lemma~\ref{thm:Stickelberger} this means that $\Disc{f-\alpha}$ is
a nonsquare in $\mathbb{F}_{q^{d^{n-1}}}$. Now, taking the norm over $\F[q]$
and using Lemma~\ref{lem:resultant_discriminant}, we get
\begin{equation*}
\begin{split}
  \Nm_{q^{d^{n-1}}|q} &\Disc{f-\alpha}\\
  &=\prod_{\substack{\alpha\in\mathbb{F}_{q^{d^{n-1}}}\\ f^{(n-1)}(\alpha)=0}} \Disc{f-\alpha}=\prod_{\substack{\alpha\in\mathbb{F}_{q^{d^{n-1}}}\\ f^{(n-1)}(\alpha)=0}} C_{f} \Res{f-\alpha}{f'}\\
  &=C_f^{d^{n-1}} \Res{\prod_{\substack{\alpha\in\mathbb{F}_{q^{d^{n-1}}}\\ f^{(n-1)}(\alpha)=0}} (f-\alpha)}{f'}\\
  &=C_f^{d^{n-1}}\Res{\frac{f^{(n-1)}(f)}{A}}{f'}=A^{-k}C_f^{d^{n-1}}\Res{f^{(n)}}{f'},
  \end{split}
\end{equation*}
where $C_f$ is defined by Lemma~\ref{lem:resultant_discriminant}, $A$ is the leading coefficient of
 $f^{(n-1)}$ and
$\Nm_{q^{d^{(n-1)}}|q}$ is the norm map from $\mathbb{F}_{q^{d^{n-1}}}$
to $\mathbb{F}_q$.

As the norm $\Nm_{q^{d^{n-1}}|q}$ maps nonsquares to nonsquares,
we obtain that $ A^{-k} C_f^{d^{(n-1)}} \Res{f^{(n)}}{f'}$ is a nonsquare, and  taking into account
that $A=a_{d}^{\frac{d^{n}-1}{d-1}}$ and 
the parity of the exponents involved,  the result follows. The case of odd $d$ can be treated in a similar way.
\end{proof}

Theorem~\ref{thm:criterium}
 is interesting because it gives a method for testing the
stability of a polynomial. Lemma~\ref{lem:resultant_eval_root} says that the
resultant is just the evaluation of $f^{(n)}$ in the roots of $\ff '$ multiplied by some constants. Taking into account this fact,
the quadratic character of $\principal{\ff}$ and the exponents
which are involved in Theorem \ref{thm:criterium}, we have the following characterisation.

\begin{corollary}\label{cor:condition}
 Let $q=p^s$, $p$ an odd prime, and $\ff\in\FX$ a stable polynomial 
 of degree $d\geq  2$ with leading coefficient $a_d$, non constant derivative 
 $\ff'$, $\deg \ff'=k\le d-1$ and $a_{k+1}$ the coefficient of $X^{k+1}$ in $f$. 
 Let $\gamma_i$, $i=1,\ldots,k,$ be the roots of the derivative
 $\ff'$. Then the following hold:
 \begin{enumerate}
 \item if $d$ is even, then
   \begin{equation}\label{eq:even orbit}
     \cS_1=
     \left\{\principal{\ff}^k  \prod_{i=1}^{k} \f{\n}{\gamma_i}\  \mid \ n>1\right\}
     \ \bigcup\ \left\{\ (-1)^\frac{d}{2}\principal{\ff}^k\prod_{i=1}^{k}
       \ff({\gamma_i})\ \right\}
   \end{equation} 
   contains only nonsquares in $\F$;
 \item if $d$ is odd, then
   \begin{equation}\label{eq:odd orbit}
     \cS_2=\left\{\ (-1)^{\frac{(d-1)}{2}+k}(k+1)a_{k+1}
       \principal{\ff}^{(n-1)k+1}\prod_{i=1}^{k} \f{\n}{\gamma_i}\  \mid \ n\ge1\right\}
   \end{equation} 
   contains only  squares in $\F$.
 \end{enumerate}
\end{corollary}
\begin{proof}
The result follows directly from Theorem~\ref{thm:criterium}
and Lemma~\ref{lem:resultant_eval_root}.
\end{proof}

We note that the converse
of Corollary~\ref{cor:condition} is not true. Indeed, take
any $d$ with $\gcd(d,\q-1)=\gcd(d,\p)=1$, $\F$ an extension
of even degree of $\F[\p]$ and $a_0$ a quadratic residue in $\F$. 
Let us consider the polynomial
$\ff{(X)}=(X-a_0)^{d}+a_0\in\FX$. It is straightforward to see that
$\f{n}{X}=(X-a_0)^{d^n}+a_0$ and that the set~\eqref{eq:odd orbit} is
\begin{equation*}
\{(-1)^{\frac{d-1}{2}}\,d\, a_0^{d-1}\}.
\end{equation*}

We note that the polynomial $\ff$ is reducible. 
Indeed, let the integer $1\le e\le q-1$ be such that $ed=1 \pmod {q-1}$. 
Then $(a_0^e)^d=a_0$, and thus $-a_0^e+a_0$ is a root of $f$. 
On the other hand, since $-1$ and $d$ are squares in
$\F$ because both elements belong to $\F[\p]$ and $\F$ is an
extension of even degree, the set~\eqref{eq:odd orbit}
contains only squares. 

We finish this section by showing that, when the derivative $f'$ of the stable polynomial $f$ is irreducible,
the sets~\eqref{eq:even orbit} and~\eqref{eq:odd orbit} are  defined
by  a short sequence of initial elements.
The proof follows exactly the same lines as in the proof of~\cite[Theorem 1]{Ostafe09}. Indeed, assume $\deg f'=k$ and $\gamma_1,\ldots,\gamma_k\in\mathbb{F}_{q^k}$ are the roots of $f'$. Using Corollary~\ref{cor:condition} we see that the sets ~\eqref{eq:even orbit} and~\eqref{eq:odd orbit} contain only nonsquares and squares, respectively, and thus, the problem reduces to the cases when  $\f{\n}{\gamma_1}\ldots\f{\n}{\gamma_k}$ are either all squares or all nonsquares for any $n\ge 1$. It is clear that, when $f'$ is irreducible, taking into account that
$\gamma_i=\gamma_1^{q^{i}}$, $i=1,\ldots,k-1$, we get for every $1\le n\le N$,
\begin{equation*}
\begin{split}
\f{\n}{\gamma_1}\ldots\f{\n}{\gamma_k} &=
\f{\n}{\gamma_1}\ldots\f{\n}{\gamma_1^{\q^{k-1}}}\\
&=\f{\n}{\gamma_1}\ldots\f{\n}{\gamma_1}^{\q^{k-1}}\ =
\Nm_{q^{k}|q}\f{\n}{\gamma_1}.
\end{split}
\end{equation*}
Applying now the same technique with multiplicative character sums as in~\cite[Theorem 1]{Ostafe09} (as the argument does not depend on the degree of the polynomial $f$), we obtain the following estimate:
\begin{theorem}
\label{thm:UB} For any odd $q$ and  any stable polynomial
$f \in\FX$ with irreducible derivative $f'$, $\deg f'=k$, there exists
$$
N= O\(q^{3k/4}\)
$$ such that  for  the sets~\eqref{eq:even orbit} and~\eqref{eq:odd orbit}
we have
\begin{equation*}
\begin{split}
  \cS_1&= \left\{\principal{\ff}^k  \prod_{i=1}^{k} \f{\n}{\gamma_i}\ \mid \ 1< n \le N\right\}\
\bigcup\ \left\{\
(-1)^\frac{d}{2}\principal{\ff}^k\prod_{i=1}^{k}
\ff({\gamma_i})\right\};\\
\cS_2& =\left\{\ (-1)^{\frac{(d-1)}{2}+k}(k+1)a_{k+1}\principal{\ff}^{(n-1)k+1}\prod_{i=1}^{k} \f{\n}{\gamma_i}\  \mid \  1\le n \le N\right\}.
\end{split}
\end{equation*}
 \end{theorem}

\section{Non-existence of certain cubic stable polynomials when $\p$=3}
\label{sec:nonExistence}
The existence of stable polynomials is difficult to prove. For $\p=2$, there are
no stable quadratic polynomials as shown in~\cite{A}, whereas
for $\p>2$, there is a big number of them as is shown in~\cite{GomezNicolas10}.
In this section, we show that for certain polynomials of degree 3, $f^{(3)}$ is a reducible
polynomial when $\p=3$. 

This result also appears in~\cite{ALOS}, but we think this approach uses new
ideas that could be of independent interest. 
For this approach, we need the following
result which can be found in~\cite[Corollary 4.6]{Menezes93}.
\begin{lemma}
\label{lem:menezes93}
  Let $q=p^s$ and $f(X)=X^{\p}-a_1 X-a_0\in\FX$ with $a_1 a_0\neq 0$. Then $f$ is irreducible over
  $\F$ if and
  only if $a_1=b^{\p-1}$ and $\Tr_{\q|p}(a_0/b^{\p})\neq 0.$
\end{lemma}
Based on this result, we can present an irreducibility criterium for polynomials
of degree 3 in characteristic $3$.
\begin{lemma}
\label{lem:appliedMenezes}
  Let $\p=3$ and $\q=3^{\s}$. Then $f(X)=X^3-a_2X^2-a_1X-a_0$ is irreducible over
  $\F$ if and only if
  \begin{enumerate}
  \item $a_1=b^2$ and $\Tr_{\q|3}(a_0/b^{3})\neq 0,$ if $a_2 =0$ and $a_1\neq 0$;
  \item $a_2^4/(a_2^2a_1^2+a_1^3-a_0a_2^3)=b^2$ and
    $\Tr_{\q|3}(1/a_2b)\neq 0,$ if $a_2 \neq 0$,
  \end{enumerate}
  where $\Tr_{\q|3}$ represent the trace map of $\F$ over $\F[3]$.
\end{lemma}
\begin{proof}
  The case $a_2=0$ is a direct application of Lemma~\ref{lem:menezes93}.
  In the other case, we take the polynomial
  \begin{multline*}
    f(X+a_1/a_2)=(X+a_1/a_2)^3-a_2(X+a_1/a_2)^2-a_1(X+a_1/a_2)-a_0=\\
    X^3-a_2X^2-a_0+a_1^2/a_2+a_1^3/a_2^3= X^3-a_2X^2+(a_1^2a_2^2+a_1^3-a_0a_2^3)/a_2^3.
  \end{multline*}
  Notice that $f(X+a_1/a_2)$ is irreducible if and only if $f(X)$ is irreducible.

  We denote $g(X)=f(X+a_1/a_2)$ to ease the notation and $g^*$
  \textit{the reciprocal polynomial} of $g$, i. e.
  \begin{equation*}
    g^*(X)=X^3g\(\frac{1}{X}\).
  \end{equation*}
  By \cite[Theorem 3.13]{LN97}, $g^*$ is irreducible if and only if $g$ is.
  Applying Lemma~\ref{lem:menezes93}, we get the result.
\end{proof}
For simplicity, we proved an irreducibility criterium for monic polynomials,
however the proof holds for
non-monic polynomials as well taking into account the principal coefficient.

Using Lemma~\ref{lem:appliedMenezes} and following the same lines as in~\cite{A}, 
we can prove now  the following result.

\begin{theorem}
\label{thm:deg3}
For any polynomial $f\in\mathbb{F}_3[X]$ of the form $f(X)=a_3X^3-a_1X-a_0$, at least 
one of the following polynomials $f,f^{(2)}$ or $ f^{(3)}$ is a reducible polynomial.
\end{theorem}
\begin{proof}
Suppose that $f,\ f^{(2)},\ f^{(3)}$ are all irreducible polynomials.
Using Lemma~\ref{lem:Capelli}, $f^{(3)}$ is irreducible if and only if
$f^{(2)}$ is irreducible over $\F$ and $f-\gamma$ is irreducible over $\F[q^9]$, where $\gamma$ is a root of $f^{(2)}$. 
Thus, the monic polynomial $h=\frac{f-\gamma}{a_3}$ is irreducible over $\F[q^9]$ 
and we can apply now Lemma~\ref{lem:appliedMenezes} from where we get that
$\Tr_{\q^9|3}(\frac{a_0-\gamma}{a_3 b^3})\neq 0$, where $b^2=a_1$ and $b\in\F[q^9]$. 

Notice that $b\in\F[q]$. Indeed, as $b$ is the root of the polynomial
$X^2-a_1$, then either $b\in\F[q]$ or $b\in\F[q^2]$. Since $b\in\F[q^9]$ 
we obtain that $b\in\F[q]$.
 Using the
properties of the trace map we obtain
\begin{equation*}
  \Tr_{q^9|3} \(\frac{a_0-\gamma}{a_3 b^3}\)=\Tr_{\q^9|3}\(\frac{-\gamma}{a_3 b^3}\),
\end{equation*}
and from here we conclude that the right hand side of the last equation is non zero. 
Using now the transitivity of the trace, see~\cite[Theorem 2.26]{LN97}, we get
\begin{equation*}
  \Tr_{\q^9|3}\(\frac{-\gamma}{a_3 b^3}\)=\Tr_{q|3}\(\Tr_{\q^9|\q}\(\frac{-\gamma}{a_3 b^3}\)\)=\Tr_{q|3}\(\frac{\Tr_{\q^9|\q}(-\gamma)}{a_3b^3}\).
\end{equation*}
Now, $f^{(2)}$ is an irreducible polynomial with roots 
$\gamma, \gamma^q,\ldots,\gamma^{q^{8}}$. Thus, $\Tr_{\q^9|\q}(\gamma)$ 
is given by the coefficient of the term $X^8$ in $f^{(2)}$, which is zero. 
This shows that $\Tr_{\q^9|3}(\gamma)=0$, which is a contradiction with the fact that
$f^{(3)}$ is irreducible.
\end{proof}

We note that Theorem~\ref{thm:deg3} cannot be extended to infinite fields. 
As in~\cite{ALOS}, let $\K = \mathbb{F}_3(T)$ be the rational function field 
in $T$ over $\mathbb{F}_3$, where $T$ is transcendental 
over $\mathbb{F}_3$. Take $f(X)=X^3+T\in \K[X]$. Then it is easy to see that
$$
f^{(n)}(X)=X^{3^n}+T^{3^{n-1}}+T^{3^{n-2}}+\cdots+T^3+T.
$$
Now from the Eisenstein criterion for function fields 
(see~\cite[Proposition~III.1.14]{Sti}, for example), 
it follows that for every $n\ge 1$, the polynomial $f^{(n)}$ is irreducible  
over $\K$. Hence, $f$ is stable. 
\section{On the number of stable polynomials}
\label{sec:numberStables}
In this section we obtain an estimate for the number of stable polynomials of
certain degree $d$. Note that, from Corollary \ref{cor:condition},
it suffices to estimate the number of nonsquares of the orbit \eqref{eq:even orbit}
for even $d$, or the number of squares of \eqref{eq:odd orbit}
for odd $d$.

For a given $d$, let
$\ff(X)=a_{d}X^{d}+a_{d-1}X^{d-1}+\cdots+a_1X+a_0\in\mathbb{F}_q[X]$
and we define
\begin{equation*}
\fa{\ff}{l}=\prod_{i=1}^{k}\ff^{(l)}({\gamma_i}),
\end{equation*}
which is a polynomial in the variables $a_{0},\ldots,\ a_{d}$ and with coefficients in
$\F$.

Following~\cite{Ostafe09}, the number of stable polynomials of degree $d$,
which will be denoted by $S_{d}$, satisfies the inequality
\begin{equation}\label{eq:sums}
S_{d}\le \frac{1}{2^{K}}\sum_{a_0\in\F}\cdots\sum_{a_{d}\in\F^*}
\prod_{l=1}^{K}(1\pm\chi(\fa{\ff}{l})),\ \forall K\in\mathbb{Z}^+,
\end{equation}

\noindent where $\chi$ is the multiplicative quadratic character of $\F$.
The sign of $\chi$ depends on $d$ and is chosen in order to count
the elements of the orbit of $\ff$ which satisfy the condition of stability.
Since the upper bound of $S_{d}$ is independent of this choice, let
us suppose from now on that $\chi$ is taken with $+$. If we expand and rearrange
the product, we obtain $2^K-1$ sums of the shape
\begin{equation*}
\sum_{a_0\in\F}\cdots\sum_{a_{d}\in\F^*}\chi
\left (\prod_{j=1}^{\mu}\fa{\ff}{l_{j}}\right ),\, 1\le l_1<\cdots < l_{\mu}\le K,
\end{equation*}

\noindent with $\mu\ge 1$ plus one trivial sum correponding to 1 in (\ref{eq:sums}).

The upper bound for $S_{d}$ will be obtained using the Weil bound for
character sums, which can be found in \cite[Lemma 1]{GomezNicolas10}. This result
can only be used when $\prod_{j=1}^{\mu}\fa{\ff}{l_{j}}$ is not a square polynomial.
The next lemmas are used to estimate the number of values for $a_{1},\ldots,a_{d}$ such that
the resulting polynomial in $a_0$ is a square.
The first lemma is a bound on the number of zeros of two multivariate polynomials.
A more general inequality is given by the Schwartz-Zippel lemma. For a proof, we
refer the reader to~\cite{Gathen99}.
\begin{lemma}
\label{multiple_roots} Let $F(Y_0,Y_1,\ldots,Y_{d}),
G(Y_0,Y_1,\ldots,Y_{d})$ be two polynomials of degree
$d_1$ and $d_2$, respectively, in $d+1$ variables with
\begin{equation*}
\gcd\left(F(Y_0,Y_1,\ldots,Y_{d}), G(Y_0,Y_1,\ldots,Y_{d})\right)=1.
\end{equation*}
Then, the number of common roots in $\F$ is bounded by
$d_1d_2\q^{d-1}$.
\end{lemma}
The next lemma gives a bound for the number of ``bad'' choices of
$a_1,\ldots,a_{d}$, that is, the number of choices of
$a_1,\ldots,a_{d}$ such that
$\prod_{j=1}^{\mu}\fa{\ff}{l_{j}}$ is a square polynomial in $a_0$.
\begin{lemma}
\label{lem:roots}
For fixed integers $l_1,\ldots, l_{\mu}$ such that $1\le l_1<\cdots < l_{\mu}\le K$,
the polynomial
\[
\prod_{j=1}^{\mu}\fa{\ff}{l_{j}}
\]
is a square polynomial in the variable $a_0$ up to a multiplicative constant
only for at most $O(d^{2K}\q^{d-1})$
choices of $a_1,\ldots,a_{d}$.
\end{lemma}
\begin{proof}
For even degree which is coprime to $\p$, we consider the polynomial $f=(X-b)^{d}+c+b$, where
$b,c$ are considered as variables. Then $f'=d(X-b)^{d-1}$ and
\begin{equation*}
f^{(n)}(b)=b+H_n(c),
\end{equation*}
where $ \deg H_n(c)= d^{n-1}$.

For odd degree, coprime to $\p$, we consider the following polynomial
$f= (X-b)^{d-1}(X-b+1)+c+b$ with the derivative
$f'=(X-b)^{d-2}(d(X-b)+d-1).$
Notice that, if the degree of this polynomial is coprime to the characteristic $\p$,
then $f'$ has two different roots $ b, b+(1-d)d^{-1}$.
Substituting these in the polynomial $f$, we get
\begin{eqnarray*}
f^{(n)}(b) &=&b+H_n(c), \\
f^{(n)}( b+(1-d)d^{-1})&=&b+L_n(c),\\
\end{eqnarray*}
where $L_n\neq H_n$ and $\deg L_n(c)=\deg H_n(c)= d^{n-1}.$
In either of the two cases, we can compute the irreducible factors
of $\Res{f^{(k)}}{\ff '}$.

When the degree is not coprime to the characteristic, 
take $f=(X-b)^{d}+(X-b)^2+c+b$ and the proof is similar to the last two cases.

This proves that the following polynomial
\[
\prod_{j=1}^{\mu}\fa{\ff}{l_{j}}
\]
is not a square polynomial as a multivariate polynomial up to a multiplicative constant.

Let
$
\prod_{j=1}^{\mu}\fa{\ff}{l_{j}}=G_1(a_0,\ldots, a_{d})^{d_1}
\cdots G_h(a_0,\ldots, a_{d})^{d_h}
$ be 
the decomposition into a product of irreducible polynomials.

Without loss of generality, $d_1$ is not even
because $\prod_{j=1}^{\mu}\fa{\ff}{l_{j}}$ is not a square of a polynomial
up to a multiplicative constant. Moreover, because $G_1$ is an irreducible
factor of the product $\prod_{j=1}^{\mu}\fa{\ff}{l_{j}}$, then 
there exists $1\le j\le \mu$ such that $G_1$ is an irreducible factor of
$\fa{\ff}{l_{j}}$, which implies that $\deg G_1\le d^{K}.$

We use $G_1(a_0,\ldots, a_{d})$ to count the number of
choices for $a_1,\ldots, a_{d}$ such that
\begin{itemize}
\item the polynomial $\prod_{j=1}^{\mu}\fa{\ff}{l_{j}}$ is a constant polynomial.
\item the polynomial $\prod_{j=1}^{\mu}\fa{\ff}{l_{j}}$ is a square polynomial
 up to a multiplicative constant.
\end{itemize}

There are at most $d^{K\mu}\q^{d-1}$ different choices of $a_1,\ldots, a_{d}$ when
the polynomial can be a constant.

Now, we consider in which cases the polynomial $\prod_{j=1}^{\mu}\fa{\ff}{l_{j}}$ is a square
of a polynomial 
and how these cases will be counted. We have the following two possible situations:
\begin{itemize}
\item $G_1^{d_1}$ is a square, nonconstant, and because $d_1$ is not
 even, then we must have that $G_1$ has
 at least one multiple root. This is only possible if
 $G_1$ and the first derivative with respect to the
 variable $a_0$ of $G_1$ have a common root.
 $G_1$ is an irreducible polynomial, so Lemma~\ref{multiple_roots} applies.
 We remark that the first derivative is a nonzero polynomial. Otherwise
 $G_1$ is a reducible polynomial.
This can only happen in $(\deg G_1)(\deg G_1-1)\q^{d-1}$ cases.

\item $G_1$ and $G_j$ have a common root for some $1\le j\le h$. In this case, using the same argument, there are  at most
 $(\deg G_1)(\deg G_j)\q^{d-1}$
 possible values for $a_1,\ldots, a_{d}$ where it happens.
\end{itemize}
\end{proof}

Now we are able to find a bound for  $S_{d}$, the number of
stable polynomials of degree $d$.



\begin{theorem}
 The number of stable polynomials $\ff\in\FX$  of degree $d$ is
 $O(q^{d+1-1/\log(2d^2)})$.
\end{theorem}

\begin{proof}
The trivial summand of~\eqref{eq:sums} can be bounded by $O(q^{d+1}/2^K)$.
For the other terms, we can use the Weil bound,
as is given in \cite[Lemma 1]{GomezNicolas10}, for those polynomials which
are nonsquares. Since these polynomials have degree at most 
$d^{K}$ in the indeterminate $a_0$ (see the proof of Lemma~\ref{lem:roots}),
we obtain $O(d^K q^{d+1/2})$ for this part.
For the rest, that is, the square polynomials, we can use the trivial bound.
Thus, from Lemma~\ref{lem:roots}, we get $O(d^{2K} q^{d})$.
Then,
\[
S_{d}=
O(q^{d+1}/2^{K}+d^Kq^{d+1/2}+
d^{2K}q^{d}).
\]
Choosing $K=\lceil(\log q/\log (2d^2))\rceil$ the result follows.
\end{proof}

%
\end{document}